\documentclass[11pt, reqno]{amsart}
\usepackage{hyperref}
\hypersetup{
	colorlinks=true,
	citecolor=blue,
	linkcolor=blue,
	filecolor=magenta,      
	urlcolor=cyan,
}
\usepackage[capitalize]{cleveref}
\usepackage{fullpage}
\usepackage{amsmath, amsthm, amssymb, graphicx, dsfont, stmaryrd, extarrows, enumitem}
\usepackage[T1]{fontenc}
\usepackage{textcomp}
\usepackage{lmodern}
\usepackage{microtype} 
\usepackage{amscd}
\usepackage{mathrsfs}
\usepackage{tikz-cd}
\usetikzlibrary{matrix,patterns}
\usepackage[colorinlistoftodos]{todonotes}
\usepackage{color}
\usepackage{bbm}
\usepackage{verbatim}
\usepackage[mathscr]{euscript}
\usepackage{bm}

\numberwithin{equation}{section}
\theoremstyle{plain}
\newtheorem{thm}[equation]{Theorem}
\newtheorem*{thm*}{Theorem}

\newtheorem{prop}[equation]{Proposition}
    
\newtheorem{cor}[equation]{Corollary}       
\newtheorem{lem}[equation]{Lemma}

\theoremstyle{definition} 
\newtheorem{defn}[equation]{Definition} 

\newtheorem{ex}[equation]{Example}

\newtheorem{rem}[equation]{Remark}

\newcommand{\mrm}{\mathrm}
\renewcommand{\t}{\text}
\newcommand{\ti}{\textit}
\newcommand{\tb}{\textbf}
\newcommand{\Mod}[1]{\mrm{Mod}(#1)}
\renewcommand{\mod}[1]{\mrm{mod}(#1)}
\newcommand{\mc}{\mathcal}
\newcommand{\Hom}{\mathrm{Hom}}
\newcommand{\rlim}{\varinjlim}
\newcommand{\llim}{\varprojlim}
\newcommand{\mf}[1]{\mathfrak{#1}}
\newcommand{\cm}[1]{\mathsf{CM}(#1)}
\newcommand{\lcm}[1]{\rlim\cm{#1}}
\newcommand{\bmat}{\begin{matrix}}
\newcommand{\emat}{\end{matrix}}
\newcommand{\spec}[1]{\mrm{Spec}(#1)}
\newcommand{\pspec}[1]{\mrm{pSpec}(#1)}
\newcommand{\mbb}[1]{\mathbb{#1}}
\newcommand{\qz}{\mbb{Q}/\mbb{Z}}

\usepackage{setspace}
\begin{document}

\title{Cotilting with Balanced Big Cohen-Macaulay modules}
\author{Isaac Bird}
\address{Department of Algebra, Faculty of Mathematics and Physics, Charles University in Prague, Sokolovsk\'{a} 83, 186 75 Praha, Czech Republic}
\email{bird@karlin.mff.cuni.cz}

\maketitle

\begin{abstract}
Over $d$-dimensional Cohen-Macaulay rings with a canonical module, $d$-cotilting classes containing the maximal and balanced big Cohen-Macaulay modules are classified. Particular emphasis is paid to the direct limit closure of the balanced big Cohen-Macaulay modules, and the class of modules of depth $d$, which are shown to respectively be the smallest and largest such cotilting classes. Considerations are then given to the interplay between local cohomology, canonical duality and cotilting modules for the class of Gorenstein flat modules over Gorenstein local rings.
\end{abstract}

\section{Introduction}
Over Cohen-Macaulay local rings admitting canonical modules, a classic result of Auslander and Buchweitz shows that every finitely generated module admits a maximal Cohen-Macaulay cover and an envelope by a finitely generated module of finite injective dimension (see \cite{lw} for a modern exposition). In \cite{holm} Holm proved that, over the same rings, this result extends beyond the finitely generated case: there is a perfect hereditary cotorsion pair $(\lcm{R},\lcm{R}^{\perp})$ in the category of all modules, meaning that every $R$-module has a cover by $\lcm{R}$, the class of modules obtained as the direct limit closure of the maximal Cohen-Macaulay modules, $\cm{R}$. 

Using these results, it was shown in \cite{bird} that the class $\lcm{R}$ is the smallest definable class containing the maximal Cohen-Macaulay modules. This class was then contrasted to the class $\mc{D}$, consisting of all modules of depth at least equal to $\t{dim}\,R$, which is another definable class whose finitely generated modules are the maximal Cohen-Macaulay modules. It was also shown that over any Cohen-Macaulay ring $(\mc{D},\mc{D}^{\perp})$ is a perfect hereditary cotorsion pair.

In this paper further properties of these cotorsion pairs are investigated, particularly from the perspective of cotilting classes. These are the classes which appear as the left hand side of the cotorsion pair $(\,^{\perp_{\infty}}C,(\,^{\perp_{\infty}}C)^{\perp_{\infty}})$ for a cotilting module $C$, referred to as the cotilting class induced by $C$. The classification of cotilting classes over commutative noetherian rings was given in \cite{apst},  establishing a bijection between such classes and specific chains of subsets of the prime spectrum. This line of investigation, as well as the title of the paper, is justified by the first main result.

\begin{thm*}[\cref{bbcmcotilting}]
	Let $R$ be a $d$-dimensional Cohen-Macaulay ring with a canonical module. Then the class $\lcm{R}$ is $d$-cotilting. Moreover, it is the smallest $d$-cotilting class containing the balanced big Cohen-Macaulay modules, and every cotilting module inducing $\lcm{R}$ is a balanced big Cohen-Macaulay module.
\end{thm*}

Subsequently, an explicit description of the corresponding subsets of $\spec{R}$ is given. As a corollary we obtain a complete description of the associated primes of the balanced big Cohen-Macaulay modules, and all of their cosyzygies, which is a partial extension of a result of Sharp. As in \cite{bird}, the comparison is made to $\mc{D}$, and the following complementary result is obtained.

\begin{thm*}[\cref{depthcotilting}]
	Let $R$ be a $d$-dimensional Cohen-Macaulay ring. Then $\mc{D}$, the class of modules of depth at least $d$, is $d$-cotilting. Moreover, it is the largest $d$-cotilting class (containing the balanced big Cohen-Macaulay modules). Every cotilting module inducing $\mc{D}$ has finite depth, equal to $d$.
\end{thm*}

By comparing the two theorems,  a complete characterisation of cotilting classes containing $\cm{R}$ and the balanced big Cohen-Macaulay modules in terms of chains of subsets of $\spec{R}$ is provided. As a consequence, it can be seen that whenever $R$ has a canonical module and dimension $d>1$, there are an abundance of cotilting classes between $\lcm{R}$ and $\mc{D}$, and all such classes are definable. This further reinforces the difference between the two as discussed in \cite{bird}.

In the final section, attention is given to the case when the ring itself is $d$-Gorenstein, where $\lcm{R}$ coincides with the class of Gorenstein flat modules. Particular phenomena occur within this setting. For example, in \cref{flatcotilting} an explicit flat cotilting module is constructed, whose associated cotilting class is the class of Gorenstein flat modules. Following this, attention is paid to functors which naturally appear in the study of maximal Cohen-Macaulay modules: canonical duality and local cohomology. This cumulates in the following theorem.

\begin{thm*}[\cref{cancot}]
	Let $R$ be a $d$-dimensional complete Gorenstein local ring. If $C$ is a cotilting module inducing $\lcm{R}$, then $\Hom_{R}(C,R)$ is a partial $d$-cotilting module with a complement. Moreover $^{\perp_{\infty}}\Hom_{R}(C,R)=\,^{\perp_{\infty}}C=\lcm{R}$.
\end{thm*}

\subsection*{Acknowledgements}
I am grateful to Mike Prest for his feedback on drafts of this work, and to Tsutomu Nakamura for comments on a previous version. I would also like to thank the referee for their helpful comments, suggestions and improvements. This research was carried out at the University of Manchester, and Univerzita Karlova under the grant 20-02760Y.

\section{Preliminaries}
Let us first establish some notation: for a ring $R$, let $\Mod{R}$ denote the class of all left $R$-modules, and associate right $R$-modules with left modules over the opposite ring $R^{\circ}$; by $\mod{R}$ we mean the class of finitely presented left $R$-modules. The class of modules of injective dimension at most $n$ is denoted by $\mc{I}_{n}$, and similarly $\mc{P}_{n}$ denotes the class of modules of projective dimension at most $n$. The modules of finite injective dimension will be denoted by $\mc{I}_{<\infty}$, and similarly define $\mc{P}_{<\infty}$ for the projective case. 

Given a class of modules $\mc{A}$,  its $i$-th right Ext-orthogonal class, where $1\leq i<\omega$, is 
\[
\mc{A}^{\perp_{i}}:=\{M\in\Mod{R}:\t{Ext}_{R}^{i}(A,M)=0 \t{ for all }A\in\mc{A}\}
\]
and its left Ext-orthogonal class, $^{\perp_{i}}\mc{A}$, is similarly defined. A pair of modules $(\mc{A},\mc{B})$ is called a \ti{cotorsion pair} if $\mc{B}=\mc{A}^{\perp_{1}}$ and $\mc{A}=\,^{\perp_{1}}B$. We let $\mc{A}^{\perp_{\infty}}=\cap_{i\geq 1}\mc{A}^{\perp_{i}}$ and similarly define $\,^{\perp_{\infty}}\mc{A}$. Given a class $\mc{A}$, the class $\t{Add}(\mc{A})$ consists of all summands of coproducts of objects in $\mc{A}$, while $\t{Prod}(\mc{A})$ consists of all summands of products of objects in $\mc{A}$. Recall that a class $\mc{A}$ is \ti{resolving} if it is extension closed, contains $\mc{P}_{0}$, and if $0\to A\to B\to C\to 0$ is a short exact sequence with $B$ and $C$ in $\mc{A}$, then $A\in\mc{A}$ as well.

We now recall some key notions concerning purity. A submodule (or equivalently an embedding) $0\to L\to M$ of $R$-modules is called \ti{pure} if for every finitely presented $R$-module $A$, the induced map $\Hom_{R}(A,M)\to\Hom_{R}(A,M/L)$ is surjective in $\tb{Ab}$, the category of Abelian groups. There are several equivalent definitions of a pure embedding, as can be seen in \cite[6.4]{JL}. A class of $R$-modules is said to be \ti{definable} if it is closed under direct limits, direct products and pure submodules. Equivalently, definable classes are precisely those given by the modules that vanish on sets of functors that arise as cokernels of $\Hom(-,\varphi_{i})$ for sets of maps $\{\varphi_{i}:A_{i}\to B_{i}\}\subset\mod{R}$. Such functors are called \ti{finitely presented}, and a detailed exposition around them can be found in \cite{PSL}.

Given any class $\mc{C}$ of modules, one can form its definable closure $\langle \mc{C} \rangle$, which is the closure of $\mc{C}$ under direct limits, direct products and pure submodules. If $\mc{C}\subset \mod{R}$, then, see for example \cite[2.13]{gt}, the direct limit closure $\rlim \mc{C}$ is always closed under pure submodules and direct limits, but closure under direct products is more elusive. A classic result \cite[4.3]{CB}, which actually holds in any finitely accessible category with products, states that $\rlim\mc{C}$ is closed under products if and only if $\mc{C}$ is \ti{pre-enveloping} in $\mod{R}$, that is for every $M\in\mod{R}$ there is an object $C\in \mc{C}$ with a morphism $\phi:M\to C$ such that any other morphism $M\to C'$, with $C'\in \mc{C}$, factors through $\phi$.

\begin{ex}
The following example will be used throughout. An $R$-module $M$ is called \ti{Gorenstein flat} if there is an acyclic complex $\bm{F}$ of flat $R$-modules with $M=Z_{0}\bm{F}$ such that $E\otimes_{R}\bm{F}$ is acyclic for all $E\in\mc{I}_{0}(R^{\circ})$. The class of Gorenstein flat $R$-modules will be denoted $\mc{GF}$. Over any ring $\mc{GF}$ is closed under coproducts, while if $R$ is right coherent it is also closed under pure submodules by \cite[2.5]{hj}. It has been shown in \cite[4.13]{ss} that $\mc{GF}$ is definable over any ring if and only if it is closed under products, and this necessitates the ring being right coherent (for left $R$-modules). By and large, we will be considering $\mc{GF}$ when the ring is Iwanaga-Gorenstein, that is $R\in\mc{I}_{<\infty}$ and is left and right noetherian. In particular $\mc{GF}$ is definable over these rings, and $\mc{P}_{<\infty}=\mc{I}_{<\infty}$. If the injective dimension of $R$ is $n$, we say that $R$ is $n$-Gorenstein.
\end{ex}

In fact, over Gorenstein rings one can say more: the class $\mc{GF}$ is actually finitely accessible, and the finitely presented objects coincide with the finitely presented \ti{Gorenstein projective} modules, where the Gorenstein projective modules are the class $\mc{GP}=\,^{\perp_{1}}\mc{P}_{<\infty}$.

We now recall the key notion of a cotilting module.

\begin{defn}\label{cotiltdef}
An $R$-module $C$ is said to be cotilting provided
\begin{enumerate}
	\item $\t{Ext}_{R}^{i}(C^{\kappa},C)=0$ for all $i\geq 1$ and cardinals $\kappa$;
	\item $C\in\mc{I}_{<\infty}$;
	\item For any injective cogenerator $E$ in $\Mod{R}$, there is an exact sequence $0\to X_{t}\to \cdots \to X_{1}\to X_{0}\to E\to 0$ where $X_{i}\in \t{Prod}(C)$ for all $0\leq i\leq t$.
\end{enumerate}
If the injective dimension of $C$ is equal to $n<\infty$, $C$ is said to be $n$-cotilting.
\end{defn}

Associated to any cotilting module $C$ is a cotorsion pair $(^{\perp_{\infty}}C, (^{\perp_{\infty}}C)^{\perp})$, which will be called the cotorsion pair induced by $C$, while $^{\perp_{\infty}}C$ is the cotilting class associated to, or induced by $C$. 

Given an arbitrary class of modules, the following lemma shows it is possible to determine whether it is a cotilting class (induced by some cotilting module).

\begin{lem}\cite[6.1]{bazz}\label{cotiltclassprops}
Let $\mc{C}$ be a class of $R$-module, then the following are equivalent, where $1\leq n<\omega$:
\begin{enumerate}
\item $\mc{C}$ is an $n$-cotilting class;
\item $\mc{C}$ is definable, resolving and $\mc{C}^{\perp_{\infty}}\subset\mc{I}_{\leq n}$;
\item $\mc{C}$ is definable, resolving and for any exact sequence $0\to X\to C_{n-1}\to\cdots\to C_{1}\to C_{0}$ with $C_{i}\in\mc{C}$, then $X\in\mc{C}$.
\end{enumerate}
\end{lem}

An exact sequence of the form that appears in the \cref{cotiltclassprops}.(3) is called an $n$-\ti{submodule}. 

\begin{ex}\label{gfex}
Returning to the above example, the class $\mc{GF}$ is $n$-cotilting if and only if $R$ is an $n$-Gorenstein ring by \cite[3.4]{aht}. 
\end{ex}

Before turning our attention to the specifics of cotilting theory over commutative noetherian rings, we recall some facts about cotilting that hold in total generality. If $\mc{C}$ is a cotilting class induced by a cotilting module $C$, for each $i\geq 0$ define
\[
\mc{C}_{(i)}:= \,^{\perp_{\infty}}\Omega^{-i}(C)=\{M:\t{Ext}_{R}^{j}(M,C)=0 \t{ for all }j>i\}
\]
where $\Omega^{-i}(C)$ is the $i$-th cosyzygy of $C$, that is the image of $f_{i}:E^{i-1}\to E^{i}$ in a minimal injective resolution of $C$. If $\mc{C}$ is $n$-cotilting, then $\mc{C}_{(i)}$ is $(n-i)$-cotilting for $0\leq i\leq n$ by \cite[15.13]{gt}, and there is a chain of inclusions $\mc{C}=\mc{C}_{(0)}\subseteq \mc{C}_{(1)}\subseteq\cdots\subseteq \mc{C}_{(n-1)}\subseteq \mc{C}_{n}=\Mod{R}$.

From this point on, all rings will be commutative and noetherian. If $M$ is an $R$-module, then a prime $\mf{p}\in\spec{R}$ is \ti{associated} to $M$ if there is an element $m\in M$ such that $\mf{p}=\t{Ann}(M)=\{r\in R:rx=0\}$. Given a class $\mc{A}\subseteq\Mod{R}$, let $\t{Ass}(\mc{A})=\cup_{A\in\mc{A}}\t{Ass}(A)\subset\t{Spec}(R)$, where $\t{Ass}(A)$ is the set of associated primes of $A$. 

For a module $M$ over a commutative noetherian ring, the injective module appearing in the $i$th entry of $M$'s minimal injective resolution is of the form
\[
E^{i}(M)=\bigoplus_{\mf{p}\in\t{Spec}(R)}E(R/\mf{p})^{(\mu_{i}(\mf{p},M))}.
\]
The $\mu_{i}(\mf{p},M)$ are called the \ti{Bass invariants} of $M$, which are determined by the formula
\[
\mu_{i}(\mf{p},M)=\t{dim}_{k(\mf{p})}\t{Ext}_{R}^{i}(R/\mf{p},M)_{\mf{p}};
\]
see \cite[\S 9.2]{rha} for proofs. 

For all $n\geq 0$, the $n$-cotilting classes over commutative noetherian rings were completely classified in \cite{apst}, and the correspondence used in the classification will be of significant utility in the forthcoming sections. The authors define a sequence $(X_{0},\cdots,X_{n-1})$ of subsets of $\t{Spec}(R)$ to be \ti{characteristic} provided
\begin{enumerate}
\item $X_{i}$ is generalisation closed for all $i<n$, that is if $\mf{q}\subset \mf{p}$ with $\mf{p}\in X_{i}$ then $\mf{q}\in X_{i};$
\item $X_{0}\subseteq X_{1}\subseteq \cdots \subseteq X_{n-1}$;
\item $\t{Ass}\,\Omega^{-i}(R)\subseteq X_{i}$ for all $i<n$.
\end{enumerate}
Given a characteristic sequence 
\begin{equation}\label{charseq}
\bm{X}=(X_{0},\cdots,X_{n-1}),
\end{equation} 
define a class of modules
\begin{equation}\label{charseqclass}
\mc{C}_{\bm{X}}=\{M\in \Mod{R}:\mu_{i}(\mf{p},M)=0 \t{ for all $i<n$ and $\mf{p}\in\t{Spec}(R)\setminus X_{i}$}\}.
\end{equation}

Conversely, given an $n$-cotilting class $\mc{C}$, we can consider the sequence of subsets of $\t{Spec}(R)$ given by $(\t{Ass}\,\mc{C}_{(0)},\t{Ass}\,\mc{C}_{(1)},\cdots,\t{Ass}\,\mc{C}_{(n-1)})$.

The classification result is as follows:

\begin{thm}{\cite[16.19]{gt}}\label{bijection}
	The assignments $\bm{X}\mapsto \mc{C}_{\bm{X}}$ and 
	$\mc{C}\mapsto (\mrm{Ass}\,\mc{C}_{(0)},\mrm{Ass}\,\mc{C}_{(1)},\cdots,\mrm{Ass}\,\mc{C}_{(n-1)})$ provide a mutually inverse bijection between $n$-cotilting classes and characteristic sequences in $\mrm{Spec}(R)$.
\end{thm}

By unravelling \cref{charseqclass}, we can see that the modules $M$ in $\mc{C}_{\bm{X}}$ are precisely those which have a minimal injective resolution of the form

\begin{align}\label{injres}
\begin{split}
0 \rightarrow & \bigoplus_{\mf{p}\in X_{0}} E(R/\mf{p})^{(\mu_{0}(\mf{p},M))} \rightarrow \bigoplus_{\mf{p}\in X_{1}}E(R/\mf{p})^{(\mu_{1}(\mf{p},M))} \rightarrow \cdots \\
\cdots & \rightarrow\bigoplus_{\mf{p}\in X_{n-1}}E(R/\mf{p})^{(\mu_{n-1}(\mf{p},M))} \rightarrow\bigoplus_{\mf{p}\in \spec{R}}E(R/\mf{p})^{(\mu_{n}(\mf{p},M))}  \rightarrow \cdots
\end{split}
\end{align}

Moreover, given an $n$-cotilting class $\mc{C}$, a prime $\mf{p}$ is in $X_{i}$ if and only if there is an $M\in\mc{C}$ with $\mu_{i}(\mf{p},M)\neq 0$ and $\mu_{j}(\mf{p},M)=0$ for all $j<i$. See, for example \cite[\S4]{sth}.

Given a commutative noetherian ring $R$, an ideal $\mf{a}$ and an $R$-module $M$, the \ti{grade} of $\mf{a}$ on $M$ is defined to be $\t{gr}(\mf{a},M):=\inf\{n\geq 0:\t{Ext}_{R}^{n}(R/\mf{a},M)\neq 0\}$ and say it is $\infty$ if no such integer exists; dually define the \ti{tor-grade} of $\mf{a}$ on $M$ via $\t{t-gr}(\mf{a},M)=\inf\{n\geq 0:\t{Tor}_{n}^{R}(R/\mf{a},M)\neq 0\}$. If $\mf{a}$ is generated by $n$ elements, \cite[6.1.8]{strooker} tells us that $\t{gr}(\mf{a},M)$ is finite if and only if $\t{t-gr}(\mf{a},M)$ is finite, in which case $\t{gr}(\mf{a},M)+\t{t-gr}(\mf{a},M)\leq n$. When $(R,\mf{m})$ is local, then $\t{gr}(\mf{m},M)$ is called the \ti{depth} of $M$ and $\t{t-gr}(\mf{m},M)$ is called the \ti{width} of $M$. 

For a finitely generated module $M$, a classic result, \cite[1.2.5]{bh}, of Rees shows that $\t{gr}(\mf{a},M)$ is nothing other than the common length of all maximal $M$-sequences in $\mf{a}$, where $\bm{x}=x_{1},\cdots,x_{n}\subset \mf{a}$ is an $M$\ti{-sequence} if 

\begin{enumerate}
\item the multiplication map $x_{i}\cdot :M/(x_{1},\cdots,x_{i-1})M\to M/(x_{1},\cdots,x_{i-1})M$ is injective for all $i\leq n$, and 
\item $M/\bm{x}M\neq 0$.
\end{enumerate} 

The notion of an $M$-sequence does not depend on the finite generation of $M$, so we use the same terminology for all modules. If only the first condition of the preceding definition holds, we say $\bm{x}$ is a \ti{weak $M$-sequence}. 

Nakayama's lemma shows that for finitely generated modules weak $M$-sequences coincide with $M$-sequences, but this is not true for arbitrary modules. The following definition is central to our study.

\begin{defn}
Let $(R,\mf{m},k)$ be a commutative noetherian local ring of krull dimension $d$. A finitely generated $R$-module $M$ is \ti{maximal Cohen-Macaulay} if $\t{depth}(M)=d$. If $R$ is a maximal Cohen-Macaulay module over itself, then $R$ is called a Cohen-Macaulay ring.
\end{defn}

We let $\cm{R}$ denote the full subcategory of $\mod{R}$ consisting of all maximal Cohen-Macaulay $R$-modules. An equivalent definition of a finitely generated module being maximal Cohen-Macaulay is that every (and hence any) system of parameters is an $M$-sequence, see \cite[21.9]{eisenbud}. This definition was used by Hochster to introduce the subclass of all modules that will be the centre of our investigation.

\begin{defn}
An $R$-module is a \ti{balanced big Cohen-Macaulay module} if every (or any) system of parameters is an $M$-sequence.
\end{defn}

We will denote the class of balanced big Cohen-Macaulay modules by $\mathsf{BBCM}(R)$. Clearly a finitely generated balanced big Cohen-Macaulay module is nothing other than a maximal Cohen-Macaulay module. 

Of particular interest to us are the Cohen-Macaulay rings that admit a \ti{canonical module} $\Omega$, which is a maximal Cohen-Macaulay module such that 
\[ 
\t{dim}_{k}\,\t{Ext}_{R}^{i}(k,\Omega)=\delta_{i,\t{dim}\,R}.
\]
Such modules are unique up to isomorphism, and a Cohen-Macaulay ring admits a canonical module if and only if it is the homomorphic image of a Gorenstein local ring. In particular Gorenstein local rings admit canonical modules, which are isomorphic to the ring itself. More information, including proofs, can be found at \cite[\S 3]{bh}.

\section{Cotilting with balanced big Cohen-Macaulay modules}
For this section, assume that $(R,\mf{m},k)$ is a Cohen-Macaulay ring of krull dimension $d$ that admits a canonical module $\Omega$. Over such a ring, Holm described the definable closure of $\cm{R}$ in \cite{holm}, and we state his result for reference purposes.

\begin{lem}\cite{holm}\label{holm}
Let $R$ be a Cohen-Macaulay ring with a dualising module $\Omega$. Then the following are equivalent for an $R$-module $M$:
\begin{enumerate}
\item $M\in\lcm{R}$;
\item every system of parameters for $R$ is a weak $M$-sequence;
\item $M$ is a Gorenstein flat module when viewed over the trivial extension $R\ltimes\Omega$;
\item $\mrm{Tor}_{i}^{R}(R/\bm{x},M)=0$ for all $i>0$ and $R$-sequences $\bm{x}$.
\end{enumerate}
\end{lem}

Moreover, \cref{holm}.(4) provides the set of finitely presented functors which yield the definable category $\lcm{R}$.

The following lemma is essentially an immediate corollary to \cref{holm}, by observing that there is a chain of inclusions $\langle\cm{R}\rangle\subseteq \langle\mathsf{BBCM}(R)\rangle\subseteq \lcm{R}$. 

\begin{lem}
The definable closure of the balanced big Cohen-Macaulay modules is $\rlim \cm{R}$.
\end{lem}

It was shown in \cite{bird}, using Holm's result, that the balanced big Cohen-Macaulay modules are precisely the modules in $\rlim \cm{R}$ of finite depth (which is necessarily equal to $d$ as $\t{Ext}_{R}^{i}(k,-)$ preserves direct limits).

For brevity, we define
\[
H_{(i)}=\{\mf{p}\in\spec{R}:\t{height}(\mf{p})\leq i\}.
\]
We are now in a position to state the first result about cotilting and the class $\lcm{R}$.

\begin{thm}\label{bbcmcotilting}
	Let $R$ be a $d$-dimensional Cohen-Macaulay ring admitting a canonical module.
	\begin{enumerate}[label={(\arabic*)}]
		\item The class $\lcm{R}$ is $d$-cotilting,
		\item The characteristic sequence corresponding to $\lcm{R}$ is 
		\[(H_{(0)},H_{(1)},\cdots, H_{(d-1)}).\]
		
		\item The corresponding sequence of cotilting classes is \[([\lcm{R}]_{(0)},[\lcm{R}]_{(1)},\cdots, [\lcm{R}]_{(d-1)}),\] where
		\[[\lcm{R}]_{(i)}=\{M\in\mrm{Mod}(R):\mrm{Tor}_{j}^{R}(R/(\bm{x}),M)=0 \t{ for all $R$-sequences $\bm{x}$ and $j>i$}\}.\]
		
		\item Every cotilting module for $\lcm{R}$ is a balanced big Cohen-Macaulay module.
	\end{enumerate}
\end{thm}

For improved legibility, we partition the proof of the theorem into components corresponding to the constituent parts. This, at times, allows for immediate comment on the relation of the theorem to extant results in the literature. Naturally, we initially prove the first claim.

\begin{proof}[Proof of \cref{bbcmcotilting}.(1)]
As described above, \cref{holm} tells us that $\lcm{R}$ is definable, so by \cref{cotiltclassprops} it is sufficient to show that $\lcm{R}$ is resolving and closed under $d$-submodules. 

Firstly, it is clear from \cref{holm} that the class is extension closed and contains the projective modules, as said modules are Cohen-Macaulay. If $0\to A\to B\to C\to 0$ is a short exact sequence with $B,C\in\lcm{R}$, then by applying $R/\bm{x}\otimes_{R}-$ and noting, again by \cref{holm}, that $\t{Tor}_{i}^{R}(R/\bm{x},C)=0$ for all $i>0$ and $R$-sequences $\bm{x}$, and likewise for $B$, we see that the class is also resolving. 

To show closure under $d$-submodules, consider the trivial extension $R\ltimes \Omega$, which is a Gorenstein local ring, also of Krull dimension $d$, and the associated functors $Z:\Mod{R}\to\Mod{R\ltimes\Omega}$ and $U:\Mod{R\ltimes \Omega}\to \Mod{R}$ as described in \cite{fgr}. By \cref{holm} we may identify $\lcm{R}$ with the $R$-modules $M$ such that $Z(M)\in\mc{GF}(R\ltimes\Omega)$. As illustrated in \cref{gfex}, the class $\mc{GF}(R\ltimes\Omega)$ is cotilting, and it is in fact $d$-cotilting: there is an inclusion $\mc{GP}\subset \mc{GF}$ in $\Mod{R\ltimes\Omega}$ and this induces a reverse inclusion $\mc{GF}^{\perp}\subset \mc{GP}^{\perp}$, but by definition $\mc{GP}^{\perp}=\mc{I}_{d}$ over a $d$-Gorenstein ring. 

In particular, if  $0\to X\to M_{0}\to M_{1}\to \cdots\to M_{d-1}$ is a $d$-submodule with $M_{i}\in\lcm{R}$, then applying the exact functor $Z$ gives a $d$-submodule $0\to Z(X)\to Z(\bm{M})$, with $Z(M_{i})\in\mc{GF}(R\ltimes\Omega)$. Yet, as just shown, $\mc{GF}(R\ltimes\Omega)$ is $d$-cotilting, so is therefore closed under $d$-submodules, hence $Z(X)\in\mc{GF}(R\ltimes\Omega)$, meaning $X\in\lcm{R}$.	
\end{proof}

Before proving the second claim, we note that it yields the associated primes of balanced big Cohen-Macaulay modules - the minimal primes - as well as providing the associated primes of their $0$th to $(d-1)$-st cosyzygies. Equivalently it gives a complete characterisation of balanced big Cohen-Macaulay modules in terms of the indecomposable injective modules which appear in the first $d$ terms of a minimal injective resolution. The associated primes of the balanced big Cohen-Macaulay modules are already known over any commutative noetherian local ring, and they are always minimal due to \cite[2.1]{sharp}. 

\begin{proof}[Proof of \cref{bbcmcotilting}.(2)]
As in \cref{bijection}, let $\bm{X}=(X_{0},\cdots,X_{d-1})$ denote the characteristic sequence corresponding to $\lcm{R}$, so $\lcm{R}=\mc{C}_{\bm{X}}$. We will show that $X_{i}=H_{(i)}$ for all $i<d$ by considering the Bass invariants of the modules in $\lcm{R}$. 

Let us first show that the Bass invariants of $\lcm{R}$ are completely determined by $\cm{R}$; that is if $\mu_{i}(\mf{p},M)\neq 0$ for some $M\in\lcm{R}$ and $\mf{p}\in\spec{R}$, then there is an $M_{0}\in\cm{R}$ with $\mu_{i}(\mf{p},M_{0})\neq 0$. Indeed, any $M\in\lcm{R}$ may be written as the directed colimit of a system $(M_{j})_{J}$ in $\cm{R}$ (the maps of the system are not needed, so are omitted). As $R/\mf{p}$ is finitely generated and localisation preserves direct limits, there are isomorphisms
\[
\t{Ext}_{R}^{i}(R/\mf{p},M)_{\mf{p}}\simeq \rlim_{J}\t{Ext}_{R}^{i}(R/\mf{p},M_{j})_{\mf{p}}
\]
for all $i\geq 0$. In particular, if $\mu_{i}(\mf{p},M)\neq 0$, then there is some $j\in J$ with $\mu_{i}(\mf{p},M_{j})\neq 0$.

Therefore, assume that $M\in\cm{R}$. If $\mf{p}\not\in\t{Supp}\,M$, then $M_{\mf{p}}=0$, so $\mu_{i}(\mf{p},M)=0$ for all $i\geq 0$ since
\[
\t{Ext}_{R}^{i}(R/\mf{p},M)_{\mf{p}}\simeq \t{Ext}_{R_{\mf{p}}}^{i}(k(\mf{p}),M_{\mf{p}})=0.
\]

On the other hand, if $\mf{p}\in\t{Supp}\,M$ then $M_{\mf{p}}$ is a maximal Cohen-Macaulay $R_{\mf{p}}$-module by \cite[2.1.3(b)]{bh}, so $\t{depth}_{R_{\mf{p}}}=\t{dim}\,R_{\mf{p}}=\t{ht}\,\mf{p}$. Yet, also by \cite[2.1.3(b)]{bh}, we have that $\t{grade}(\mf{p},M)=\t{depth}\,M_{\mf{p}}$. Combined, this tells us that $\t{Ext}_{R}^{i}(R/\mf{p},M)=0$ for all $i<\t{ht}\,\mf{p}$, in other words $\mu_{i}(\mf{p},M)=0$ for all $i<\t{ht}\,\mf{p}$. Therefore, for any $\mf{q}\in\spec{R}$, we have that $\mu_{j}(\mf{q},M)=0$ for every $M\in\lcm{R}$ and $j<\t{height}(\mf{q})$.

So suppose that $\mf{p}\in X_{i}$. Then, by the discussion at \cref{injres}, and the preceding part of the proof, there is a maximal Cohen-Macaulay module $M$ with $\mu_{i}(\mf{p},M)\neq 0$. If $\t{height}(\mf{p})>i$, then we must have $\mu_{i}(\mf{p},M)=0$, which is a contradiction, so $\t{height}(\mf{p})\leq i$ and thus $X_{i}\subseteq H_{(i)}$.

In order to show that $X_{i}=H_{(i)}$ it suffices to show that for every $\mf{p}\in H_{(i)}$ there is a module in $\lcm{R}$ such that $E(R/\mf{p})$ is a direct summand of the $i$th term of its minimal injective resolution. As the canonical module $\Omega_{R}$ is faithful, it is supported everywhere, but $(\Omega_{R})_{\mf{p}}\simeq \Omega_{R_{\mf{p}}}$ for all $\mf{p}\in\t{Spec}(R)$ by \cite[3.3.5]{bh}. In particular, if $\t{ht}\,\mf{p}=i$, we see that $\t{Ext}_{R_{\mf{p}}}^{i}(k(\mf{p}),(\Omega_{R})_{\mf{p}})\neq 0$, hence $\mu_{i}(\mf{p},\Omega_{R})\neq 0$. Alternatively, consider the discussion at \cite[Remark 16.13]{gt}. This concludes the proof. 
\end{proof}

We now interpose the proof of \cref{bbcmcotilting} with a corollary of \cref{bbcmcotilting}.(2), relating to the specific case when $R$ is a (not necessarily local) commutative Gorenstein ring. Over local Gorenstein rings, the class of maximal Cohen-Macaulay modules coincides with the class of finitely presented Gorenstein flat modules (which are just the finitely presented Gorenstein projectives) by \cite[11.5.4]{rha}, and therefore $\lcm{R}=\mc{GF}(R)$. In particular, we can use the above theorem to completely identify Gorenstein flat modules in terms of their Bass invariants, even in the non-local case.

\begin{cor}
	Let $R$ be a Gorenstein ring. Then the following are equivalent for an $R$-module $M$:
	\begin{enumerate}[label={(\arabic*)}]
		\item $M$ is a Gorenstein flat $R$-module,
		\item For all $\mf{p}\in\mrm{Spec}(R)$, the Bass invariant $\mu_{i}(\mf{p},M)$ is zero for all $i<\mrm{ht}\,\mf{p}$.
	\end{enumerate}
\end{cor}

\begin{proof}
Let us first observe that $M$ is a Gorenstein flat $R$-module if and only if $M_{\mf{p}}$ is a Gorenstein flat $R_{\mf{p}}$-module for all prime ideals $\mf{p}$. To show the first implication, let $M$ be a Gorenstein flat $R$-module and $\mf{p}$ be a prime ideal. For each $i\geq 0$ there is an equality $\mu_{i}(\mf{p},M)=\mu_{i}(\mf{p}R_{\mf{p}},M_{\mf{p}})$ by \cite[9.2.1]{rha}. Yet $M_{\mf{p}}$ is a Gorenstein flat $R_{\mf{p}}$-module, and since $R_{\mf{p}}$ is a Gorenstein local ring $\mc{GF}(R_{\mf{p}})=\rlim\cm{R_{\mf{p}}}$ so the implication follows from \cref{bbcmcotilting}.(2).

For the reverse implication, assume the Bass in variants of $M$ are as in the statement. We will show that $M_{\mf{p}}$ is a Gorenstein flat $R_{\mf{p}}$-module for every $\mf{p}\in\t{Spec}(R)$. If $\mf{q}$ is a prime ideal of $R$ contained in $\mf{p}$ then $(R\setminus\mf{p})\cap\mf{q}=\varnothing$, so there is an equality $\mu_{i}(\mf{q},M)=\mu_{i}(\mf{q}_{\mf{p}},M_{\mf{p}})$ for every $i\geq 0$ by \cite[9.2.1]{rha}; in particular, $\mu_{i}(\mf{q}_{\mf{p}},M_{\mf{p}})=0$ for all $i<\t{ht}\,\mf{q}$. Yet the prime ideals in $R_{\mf{p}}$ are precisely the localisations of the prime ideals of $R$ contained in $\mf{p}$, so for every $\mf{q}_{\mf{p}}\in\t{Spec}(R)$ we see $\mu_{i}(\mf{q}_{\mf{p}},M_{\mf{p}})=0$ for all $i<\t{ht}\,\mf{q}_{\mf{p}}=\t{ht}\,\mf{q}$. Since $R_{\mf{p}}$ is a Gorenstein local ring, it is Cohen-Macaulay and we can use the characteristic sequence in \cref{bbcmcotilting}.(2) to see that $M_{\mf{p}}$ is in $\rlim\cm{R_{\mf{p}}}$, so is a Gorenstein flat $R_{\mf{p}}$-module. 
Since $\mf{p}$ was arbitrary, it holds at every $\mf{p}\in\t{Spec}(R)$, so $M$ is a Gorenstein flat $R$-module.
\end{proof}

The result in the above corollary is not new, and it can be deduced from \cite[2.3.13]{christensen}. However, the tools of cotilting provide an alternative proof, hence the reason for inclusion. It is conceivable that there are non-Gorenstein commutative noetherian rings $R$ such that the Gorenstein flat $R$-modules have the Bass invariants given in the above corollary. Indeed, $R$ itself would be a Gorenstein flat module, so the localisation $R_{\mf{p}}$ would be a Cohen-Macaulay local ring at each prime ideal. However, since the above result cannot determine $\mu_{i}(\mf{p},R)$ for $i>\t{ht}\,\mf{p}$, it is not immediately deducible that $R_{\mf{p}}$ has finite injective dimension. 

Let us now return to the proof of the remaining two parts of the theorem.

\begin{proof}[Proof of \cref{bbcmcotilting}.(3)]
We proceed by induction on $i$. When $i=0$ the result is immediate from the characterisation of $\lcm{R}$. For induction assume that
\[
[\lcm{R}]_{(i)}=\{M\in\Mod{R}:\t{Tor}_{j}^{R}(R/(\bm{x}),M)=0 \t{ for all $R$-sequences $\bm{x}$ and $j>i$}\}.
\]
Let $M$ be an $R$-module and assume $M\in [\lcm{R}]_{(i+1)}$. Consider the canonical short exact sequence $0\to \Omega^{1}(M)\to P\to M\to 0$ with $P$ projective, so $P\in [\lcm{R}]_{(i)}$ as it is a cotilting class. By \cite[16.14]{gt} we see $\Omega^{1}(M)\in[\lcm{R}]_{(i)}$ as well. By dimension shifting, for every $R$-sequence $\bm{x}$ and $k\geq 1$ there is an isomorphism
\[
\t{Tor}_{k}^{R}(R/(\bm{x}),\Omega^{1}(M))\simeq \t{Tor}_{k+1}^{R}(R/(\bm{x}),M),
\]
hence $\t{Tor}_{i+1+\lambda}^{R}(R/(\bm{x}),M)=\t{Tor}_{i+\lambda}^{R}(R/(\bm{x}),\Omega^{1}(M))=0$ for every $\lambda>0$ by the induction hypothesis. Conversely, suppose $\t{Tor}_{j}^{R}(R/(\bm{x}),M)=0$ for all $j>i+1$ and $R$-sequences $\bm{x}$. By a similar dimension shifting argument we see that $\t{Tor}_{j}^{R}(R/(\bm{x}),\Omega^{1}(M))=0$ for all $j>i$, that is $\Omega^{1}(M)\in[\lcm{R}]_{(i)}$ by the inductive hypothesis. We may therefore apply \cite[16.14]{gt} again to see that $M\in[\lcm{R}]_{(i+1)}$, which proves the claim.
\end{proof}

The functors $\t{Tor}_{j}^{R}(R/\bm{x},-)$ are finitely presented for all $j\geq 0$ as the modules $R/\bm{x}$ are finitely presented for all $R$-sequence $\bm{x}$ by \cite[10.2.36]{PSL}. 

\begin{proof}[Proof of \cref{bbcmcotilting}.(4)]
Let $C$ be a cotilting module with $^{\perp_{\infty}}C=\lcm{R}$. As $C$ is cotilting, it is in $\lcm{R}$, where every module in $\lcm{R}$ has depth either equal to $\t{dim}\,R$ or $\infty$; as mentioned above, $\mathsf{BBCM}(R)=\{M\in\lcm{R}:\t{depth}(M)=\t{dim}(R)\}$. Suppose that $C$ has infinite depth, in other words $\t{Ext}_{R}^{i}(k,C)=0$ for all $i\geq 0$. Then $k$ is certainly in $^{\perp_{\infty}}C$, and is therefore a finitely presented object in $\lcm{R}$, in other words a maximal Cohen-Macaulay module. But this is absurd, since the depth of $k$ is zero as $\Hom_{R}(k,k)\simeq k$. Therefore $C$ has depth equal to $d$, so is in $\mathsf{BBCM}(R)$.
\end{proof}

As a remark, there is a step-by-step process on how to construct cotilting modules given a cotilting class in \cite{sth}. If one follows this procedure, applications of the depth lemma \cite[9.1.2.(e)]{bh}, and noting that $E(k)$ is the only indecomposable injective $R$-module of finite depth - it has zero depth - one also obtains the existence of a cotilting balanced big Cohen-Macaulay module that induces $\lcm{R}$.

\section{Depth and cotilting}

As mentioned at the start of the previous section, the class $\lcm{R}$ is the definable closure of $\cm{R}$ whenever $R$ admits a canonical module (the question of whether $\lcm{R}$ is definable over general Cohen-Macaulay rings is still open). However, there is a natural alternative definable extension of $\cm{R}$, given by the class of modules of depth at least $d=\t{dim}\,R$. Certain differences between these two classes were discussed in \cite{bird}. We will now see how cotilting, or rather the characteristic sequence associated to a cotilting class, provides an alternative way to consider distinctions between these classes, as well as providing a way to `measure' how much they differ. Before formalising this, we introduce the following notation: for $i\geq 0$ define the classes
\[
\mc{D}_{i}=\{M\in\Mod{R}:\t{depth}(M)\geq i\}=\{M\in\Mod{R}:\t{Ext}_{R}^{j}(k,M)= 0 \t{ for }j<i\}.
\]
Each $\mc{D}_{i}$ is clearly definable. We will let $\pspec{R}$ denote the punctured spectrum, that is $\spec{R}\setminus\{\mf{m}\}$. From now we will assume that $R$ is an arbitrary Cohen-Macaulay ring, not necessarily with a canonical module. Somewhat unsurprisingly, given the above discussion, we have the following theorem that is a complete analogue of \cref{bbcmcotilting}.

\begin{thm}\label{depthcotilting}
Let $R$ be an arbitrary Cohen-Macaulay ring of dimension $d$. 
\begin{enumerate}[label={(\arabic*)}]
		\item $\mc{D}_{d}$ is $d$-cotilting.
		\item The characteristic sequence for $\mc{D}_{d}$ is $(\pspec{R},\pspec{R},\cdots,\pspec{R})$.
		\item The sequence of cotilting classes corresponding to $\mc{D}_{d}$ is
		\[
		(\mc{D}_{d},\mc{D}_{d-1},\cdots,\mc{D}_{1}).
		\]
		\item Set $\mc{E}_{0}=\mrm{Add}\,\{E(R/\mf{p}):\mf{p}\in\pspec{R}\}.$ Then a cotilting module inducing $\mc{D}_{d}$ is
		\[
		\Omega_{\mc{E}_{0}}^{d}(E(k)),
		\]
		the $d$th syzygy of a minimal resolution of $E(k)$ with respect to the class $\mc{E}_{0}$. Moreover, this module has depth $d$.
	\end{enumerate}
\end{thm} 

\begin{proof}[Proof of \cref{depthcotilting}.(1)]
It is clear that $\mc{D}_{d}$ is definable, extension closed and contains $R$ as we assumed the ring was Cohen-Macaulay. Closure under kernels of epimorphisms follows from the depth lemma. We show it is also closed under $d$-submodules. Suppose
\[
\begin{tikzcd}
	0 \arrow[r] & X\arrow[r, "f_{0}"] & M_{0}\arrow[r, "f_{1}"] & M_{1}  \arrow[r] &\cdots \arrow[r] &M_{d-2}\arrow[r, "f_{d-1}"] &M_{d-1}
\end{tikzcd} 
\]
is a $d$-submodule with each $M_{i}\in\mc{D}_{d}$. This exact sequence decomposes into $d$ short exact sequences $0\to \t{im}\,f_{i}\to M_{i}\to \t{coker}\,f_{i}\to 0$, and the depth lemma shows that $\t{depth}\,\t{im}\,f_{i}\geq \t{depth}\,\t{coker}\,f_{i}+1$ for all $i$, as $\t{depth}\,M_{i}\geq d$. Yet $\t{coker}\,f_{i}\simeq \t{im}\,f_{i+1}$ for all $0\leq i<d-1$, and therefore by iterating, we have $\t{depth}\, X = \t{depth}\,\t{im}\,f_{0}\geq \t{depth}\,\t{coker}\,f_{d-1}+d$. Consequently $X\in\mc{D}_{d}$, so the class is $d$-cotilting as claimed.
\end{proof}

The proof also shows that each of the classes $\mc{D}_{i}$ is $i$-cotilting for all $0<i\leq d$. 
Before we prove \cref{depthcotilting}.(2), let us prove \cref{depthcotilting}.(3). What it shows, that $(\mc{D}_{d})_{(i)}=\mc{D}_{d-i}$, enables an easier proof of \cref{depthcotilting}.(2).

\begin{proof}[Proof of \cref{depthcotilting}.(3)]
This is very similar to the proof of \cref{bbcmcotilting}.(3), and it follows from the observation that $(\mc{D}_{d})_{(i)}=\mc{D}_{d-i}$, which can be shown by dimension shifting and induction on $i$.
\end{proof}

We now prove \cref{depthcotilting}.(2).

\begin{proof}[Proof of \cref{depthcotilting}.(2)]

Let $\bm{X}_{\mc{D}_{d}}=(X_{0},\cdots,X_{d-1})$ denote the characteristic sequence for $\mc{D}_{d}$, as detailed in \cref{bijection}. By \cref{depthcotilting}.(3), we know $(\mc{D}_{d})_{(i)}=\mc{D}_{d-i}$, so we may apply \cite[Lemma 16.15(a)]{gt}, which tells us that $\mf{p}\in X_{i}$ if and only if $E(R/\mf{p})\in \mc{D}_{d-i}$. Now, if $\mf{p}$ is a prime ideal, there are only two possibilities for the depth of $E(R/\mf{p})$: it is either zero or infinite. Let $k\to E(R/\mf{p})$ be a homomorphism, which, by the property of injective hulls, factors through $E(k)$. But by \cite[3.3.8(4)]{rha}, we have $\Hom_{R}(E(k),E(R/\mf{p}))\neq 0$ if and only if $\mf{p}=\mf{m}$. Consequently for any $\mf{p}\in \pspec{R}$, the depth of $E(R/\mf{p})$ is infinite, and is therefore in $\mc{D}_{i}$ for all $0\leq i\leq n$. Since $X_{i}=\t{Ass}(\mc{D}_{d-i})$, we see that $\pspec{R}\subseteq X_{i}$ for each $1\leq i\leq d$. Yet this is an equality, since if it were not, we would have $E(k)\in X_{i}$ for some $1\leq i\leq d$, which would tell us that $E(k)$ has depth at least $1$. We have just seen this is not possible, hence $X_{i}=\pspec{R}$.

\end{proof}

This shows that over Cohen-Macaulay rings, the class $\mc{D}_{i}$ is the largest non-trivial  $i$-cotilting class in $\Mod{R}$. Indeed, if $\mc{C}$ is an $i$-cotilting class with characteristic sequence $\bm{X}_{\mc{C}}=(X_{0},\cdots,X_{i})$, then none of the $X_{i}$ can equal to $\spec{R}$. In particular, each $X_{i}\subseteq \pspec{R}$. Hence we must have $\mc{C}\subset\mc{D}_{d}$. 

This is in contrast to the case of $\lcm{R}$ when $R$ admits a canonical module, which is the smallest $d$-cotilting class. Indeed, if $\mc{C}$ is any $d$-cotilting class, with characteristic sequence $\bm{X}_{\mc{C}}$ as in \cref{charseq}, then $X_{i}$ must contain all the associated primes of $\Omega^{-i}(R)$, which are precisely the height $i$ primes. In particular, we see $H_{(i)}\subseteq X_{i}$, for each $i$, so by the discussion at \cref{injres} it follows that $\lcm{R}\subseteq\mc{C}$.  

We can use the characteristic sequences given in \cref{bbcmcotilting}.(2) and \cref{depthcotilting}.(2) to see that there are usually a great many $d$-cotilting classes contained between $\lcm{R}$ and $\mc{D}_{d}$, let alone definable classes.

\begin{prop}
Let $R$ be a $d$-dimensional Cohen-Macaulay ring with canonical module $\Omega$. Then there are finitely many $d$-cotilting classes containing $\cm{R}$ if and only if $d=1$.
\end{prop}
\begin{proof}
	One direction is trivial: if $d=1$ then $\lcm{R}=\mc{D}_{1}$ by \cite[3.8]{bird} hence there is a single cotilting class containing $\cm{R}$, which is $\lcm{R}$. For the other direction, we prove the $d=2$ case, which is easily generalised to higher dimension. By \cite[144]{kap}, there are infinitely many height one primes, and let $\mf{p}$ be one of them. The sequence of subsets of $\spec{R}$ given by $(H_{(0)}\cup\{\mf{p}\}, H_{(1)})$ is characteristic. Indeed we know $(H_{(0)},H_{(1)})$ is characteristic by \cref{bbcmcotilting}, so $\t{Ass}\,\Omega^{-i}(R)\subseteq H_{(i)}$ for $i=0,1$, and therefore these associated primes are also in $H_{(0)}\cup\{\mf{p}\}$ and $H_{(1)}$ respectively. The only remaining condition that requires any comment is showing that $H_{(0)}\cup \{\mf{p}\}$ is generalisation closed. But since the only primes contained in $\mf{p}$ are of height zero, they are trivially also in $H_{(0)}\cup\{\mf{p}\}$. For higher dimensions, pick a prime $\mf{p}$ of given height and considers its specialisation closure $\overline{\mf{p}}=\{\mf{q}\in\spec{R}:\mf{q}\subset\mf{p}\}$. Then the sequence which has in its $i$th position $H_{(i)}\cup\overline{\mf{p}}$ is characteristic.
\end{proof}
This shows that there are an abundance of cotilting classes between $\lcm{R}$ and $\mc{D}_{d}$. 

\begin{rem}
There are alternative ways to measure the `size' of the disparity between $\lcm{R}$ and $\mc{D}_{d}$, although they are much less immediate than by comparing characteristic sequences. 

There is a correspondence between definable subcategory of $\Mod{R}$ and Serre subcategories of $\mod{\mod{R}}$, the category of functors $\mod{R}\to\tb{Ab}$ which preserve direct limits and direct products. The correspondence sends a definable category $\mc{X}$ to the Serre subcategory
\[
\{F\in\mod{\mod{R}}:\overset{\rightarrow}{F}X=0 \t{ for all }X\in\mc{X}\},
\]
where $\overset{\rightarrow}{F}$ is the unique functor $\Mod{R}\to\tb{Ab}$ extending $F$ and commuting with direct limits. The reverse direction sends a Serre subcategory $\mc{S}$ to $\{X\in\Mod{R}:\overset{\rightarrow}{F}X=0 \t{ for all }F\in\mc{S}\}$. 

If $\mc{S}_{1}$ and $\mc{S}_{2}$ are the Serre subcategories corresponding to the cotilting classes $\lcm{R}$ and $\mc{D}_{d}$ respectively, then there is a reverse inclusion $\mc{S}_{2}\subseteq\mc{S}_{1}$, where the inclusion is strict whenever $\t{dim}(R)>1$. For any functor $F\in\mc{S}_{1}\setminus\mc{S}_{2}$, there is a definable category corresponding to the Serre subcategory generated by $\mc{S}_{2}\cup\{F\}$. This will be a definable category lying between $\lcm{R}$ and $\mc{D}_{d}$, although there is no reason for it to be cotilting.

We can also consider the definable quotient category, which is the definable category corresponding to the small abelian category $\mc{S}_{1}/\mc{S}_{2}$. This category contains information about the discrepancy in the indecomposable pure injectives between $\mc{D}_{d}$ and $\lcm{R}$ (we have already partially considered this by investigating the indecomposable injective objects in each class). The downside to this approach is that the definable quotient is not a definable subcategory of $\Mod{R}$. Instead it is just the definable subcategory of some finitely accessible category with products. See \cite{krause} for more details on definable quotients. 

\end{rem}

We now return to the final section of the proof of \cref{depthcotilting}.

\begin{proof}[Proof of \cref{depthcotilting}.(3)]
This is very similar to the proof of \cref{bbcmcotilting}.(3), and it follows from the observation that $(\mc{D}_{d})_{(i)}=\mc{D}_{d-i}$, which can be shown by dimension shifting and induction on $i$.
\end{proof}

\begin{proof}[Proof of \cref{depthcotilting}.(4)]
	Firstly, any cotilting module for $\mc{D}_{d}$ must be of finite depth for the same reasons as in \cref{bbcmcotilting}.(4). To show that the module as described is a cotilting module, we follow the procedure in \cite{sth}. In brief, this states that there is a cotilting module inducing $\mc{D}_{d}$ of the form
	\[C\simeq\prod_{\spec{R}}C(\mf{p})\]
	where $\mc{C}(\mf{p}):=E(R/\mf{p})$ if $\mf{p}\in\t{Ass}\,\mc{D}_{d}=\pspec{R}$ and $C(\mf{m})$ arises in the following exact sequence
	\[
	\begin{tikzcd}
		0\arrow{r} & C(\mf{m})\arrow{r} & E_{0} \arrow{r}{\phi_{0}}&E_{1}\arrow{r}&\cdots\arrow{r}&E_{d-2}\arrow{r}{\phi_{d-2}}& E_{d-1} \arrow{r}{\phi_{d-1}} & E(k)\arrow{r} & 0
	\end{tikzcd}
	\]
	where $\phi_{d-1}:E_{d-1}\to E(k)$ is an $\mc{E}_{0}$-cover, $\phi_{i}:E_{i}\to E_{i+1}$ is an $\mc{E}_{0}$-cover of $\t{ker}\,\phi_{i+1}$ for all $i\leq d-2$ and $C(\mf{m})=\t{ker}\,\phi_{0}$. Yet this is the same as saying that $C(\mf{m})$ is the $d$th syzygy of a minimal $\mc{E}_{0}$ resolution of $E(k)$. Consequently $C$ is the direct sum of $\prod_{\pspec{R}}E(R/\mf{p})$ and $C(\mf{m})$, but the product of injectives is injective, and therefore redundant to the cotilting structure. Hence $C(\mf{m})=\Omega_{\mc{E}_{0}}^{d}(E(k))$ is a cotilting module.
\end{proof}

\section{Some phenomena over Gorenstein rings}
There is a weaker notion than a cotilting module, that of a \ti{partial cotilting module}. As opposed to satisfying all three conditions of \cref{cotiltdef}, a partial cotilting module only satisfies the first two - that is $X$ is a partial cotilting module if it has finite injective dimension and $\t{Ext}_{R}^{i}(X^{\kappa},X)=0$ for all cardinals $\kappa$ and $i>0$. Such a module is said to have a \ti{complement} if there is a module $Y$ such that $X\oplus Y$ is a cotilting module. 

\begin{ex}
	If $R$ is a complete local Gorenstein ring, then $R$ itself is a partial cotilting module: it trivially has finite injective dimension from the Gorenstein assumption, while as $R$ is noetherian it is certainly coherent and thus $R^{\kappa}$ is a flat module for every cardinal $\kappa$. Yet as $R$ is complete it is pure injective as it is matlis reflexive, and therefore  it is also cotorsion. Consequently $\t{Ext}_{R}^{1}(R^{\kappa},R)=0$. In this case the class $^{\perp_{\infty}} R$ is nothing other than $\mc{D}_{d}$, where $d=\t{dim}\,R$ (this follows from local duality \cite[3.5.8]{bh}). In particular $^{\perp_{\infty}}R$ is a cotilting class.
\end{ex}

We now formalise the above example. Note that all rings are once again assumed to be commutative and noetherian.

\begin{lem}\label{GFlatInclusion}
Let $R$ be a Gorenstein ring. If $X$ is a partial cotilting and pure-injective $R$-module, then there is always an inclusion $\mc{GF}(R)\subseteq\,^{\perp_{\infty}}X$. Moreover, if $X$ is flat, it admits a complement.
\end{lem}

\begin{proof}
Since $X$ has finite injective dimension \cite[10.2.6(3)]{rha} tells us that $\t{Ext}_{R}^{i}(M,X)=0$ for all $i\geq 1$ and finitely generated Gorenstein projective modules $M$. As $X$ was assumed to be pure-injective there are isomorphisms
\[
\t{Ext}_{R}^{i}(\rlim_{J}N_{j},X)\simeq \llim_{J}\t{Ext}_{R}^{i}(N_{j},X)\]
for every directed system $(N_{j},f_{ij})_{J}$ of $R$-modules by \cite[6.28]{gt}. In particular, we have $^{\perp_{\infty}}X$ closed under direct limits, hence $\rlim\mc{GP}^{\t{fp}}=\mc{GF}$ is contained in $^{\perp_{\infty}}X$. For the latter claim, note that since $R$ is Gorenstein \cite[6.1]{emm} shows that $R$ has finite global $\mc{GF}$-dimension. Moreover, as $\mc{GF}$ is definable, it contains $\t{Prod}(X)$ since $X$ is assumed to be flat. $\mc{GF}$ is also resolving. We may therefore apply \cite[6.1]{ks} to see that $X$ admits a complement.
\end{proof}

In general it need not be the case that, given a partial cotilting module $X$ with complement $Y$, the classes $^{\perp_{\infty}}X$ and $^{\perp_{\infty}}(X\oplus Y)$ coincide. For this, one needs the class $^{\perp_{\infty}}X$ to be closed under products (see \cite[6.3]{ks}), as in the case of the above example. We note that conditions for the existence of complements for cotilting modules were also considered in \cite[p.93]{ahc}, as the dual question to the existence of a complement for a partial tilting module.

\begin{thm}\label{flatcotilting}
Let $R$ be a commutative Gorenstein ring. Then
\[
F=\prod_{\mf{p}\in\spec{R}}\widehat{R_{\mf{p}}}
\]
is a cotilting module for $\mc{GF}(R)$, where $\widehat{R_{\mf{p}}}$ denotes the $\mf{p}$-adic completion of $R_{\mf{p}}$.
\end{thm}

Before giving the proof, note that $\widehat{R_{\mf{p}}}$ is a flat pure-injective $R$-module that is isomorphic to $\Hom_{R_{\mf{p}}}(E(\kappa(\mf{p})),E(\kappa(\mf{p})))$, and every flat and pure-injective $R$-module is a product of completions of free $R_{\mf{p}}$-modules. This can be seen in \cite[6.7]{rha}. Moreover, over such rings any flat and pure-injective module is partial cotilting for the same reason as given in the above example. In particular, we see that $\mc{GF}(R)\subseteq \,^{\perp_{\infty}}F$ for any flat and pure-injective module $F$. As we will see in the proof, it is possible for such a module to be a cotilting, not just partial cotilting, module, if and only if $\widehat{R_{\mf{p}}}$ appears as a summand of the flat module for every prime $\mf{p}$.

\begin{proof}[Proof of \cref{flatcotilting}]
Let us first show that $^{\perp_{\infty}}F=\mc{GF}(R)$. It is enough to show the inclusion $^{\perp_{\infty}}F\subseteq\mc{GF}(R)$ by \cref{GFlatInclusion}. Suppose that $M$ is in $^{\perp_{\infty}}F$; this occurs if and only if 
\[
\prod_{\mf{p}}\t{Ext}_{R}^{i}(M,\widehat{R_{\mf{p}}})=0 \t{ if and only if } \t{Ext}_{R}^{i}(M,\widehat{R_{\mf{p}}})=0\t{ for all prime ideals }\mf{p}.
\]
Yet $E(R/\mf{p})$ is an injective cogenerator in $\Mod{R_{\mf{p}}}$, and there are isomorphisms
\begin{align*}
\t{Ext}_{R}^{i}(M,\widehat{R_{\mf{p}}})&\simeq \t{Ext}_{R}^{i}(M,\Hom_{R_{\mf{p}}}(E(\kappa(\mf{p}))E(\kappa(\mf{p}))) \\
&\simeq \Hom_{R_{\mf{p}}}(\t{Tor}_{i}^{R}(M,E(R/\mf{p})),E(\kappa(\mf{p}))),
\end{align*}
where the second isomorphism can be found at \cite[2.16]{gt} and the isomorphism of $R$-modules $E(R/\mf{p})\simeq E(\kappa(\mf{p}))$ is used. In particular, we see that the first module is zero if $\t{Tor}_{i}^{R}(M,E(R/\mf{p}))=0$ for every prime ideal $\mf{p}$. Yet as $R$ is Gorenstein and commutative, it follows that $\t{Tor}_{i}^{R}(M,E)=0$ for every injective module, hence $M\in\mc{GF}(R)$ and $^{\perp_{\infty}}F=\mc{GF}(R)$. By the above discussion, to finish the proof it suffices to show the third condition of \cref{cotiltdef}. Let $F_{i}=\prod_{\t{ht}\,\mf{p}=i}\widehat{R_{\mf{p}}}$, so $F=\oplus_{0\leq i\leq \t{dim }R}F_{i}$. The description of a minimal flat resolution of injective modules over a Gorenstein ring can be found at \cite[\S 5.3]{xu}; in particular, if $E$ is an injective cogenerator for $\Mod{R}$ with minimal flat resolution $\cdots\to G^{1}\to G^{0}\to E\to 0$, then $G^{i}\in\t{Prod}(F_{i})$ by \cite[5.3.1]{xu}, and consequently each $G^{i}$ is in $\t{Prod}(F)$. This shows that the third condition of \cref{cotiltdef} holds, which finishes the proof.
\end{proof}

\begin{rem}
In a private correspondence with Michal Hrbek, I was informed that a forthcoming preprint authored him, J. \v{S}\v{t}ov\'{\i}\v{c}ek and T. Nakamura also contains the above theorem, obtained completely independently from this work, and using different techniques.   
\end{rem}

Suppose that, given any ring $R$, one has a definable category $\mc{D}$ of $R$-modules. There is a unique definable category of $\Mod{R^{\circ}}$, called the \ti{dual definable category} of $\mc{D}$, and denoted $\mc{D}^{d}$, given by the property that $\Hom_{\mbb{Z}}(M,\qz)\in\mc{D}^{d}$ if and only if $M\in\mc{D}$, and that $\mc{D}^{dd}=\mc{D}$. For example, over Gorenstein rings the classes $\mc{GF}$ and $\mc{GI}$ are both definable, and $(\mc{GF})^{d}=\mc{GI}$.  Over commutative noetherian rings, every cotilting class is the dual definable category of a tilting class, see, for example, \cite[16.21]{gt}. 

The dual notion to cotilting is tilting. An $R$-module $T$ (over any ring) is \ti{tilting} if it has finite projective dimension, $\t{Ext}_{R}^{i}(T,T^{(\lambda)})=0$ for all $i>0$ and cardinal $\lambda$, and there is an exact sequence $0\to R\to T_{0}\to\cdots\to T_{r}\to 0$, where $T_{i}\in\t{Add}(T)$. Dually to the cotilting case, there is a cotorsion pair $(\,^{\perp_{\infty}}(T^{\perp_{\infty}}),T^{\perp_{\infty}})$, and the class $T^{\perp_{\infty}}$ is called the \ti{tilting class} associated to $T$. As shown in \cite[4.2]{fintype}, all tilting classes are definable, and if $T$ is tilting with tilting class $\mc{T}$, then $\Hom_{\mbb{Z}}(T,\qz)$ is a cotilting module, whose cotilting class $\mc{C}$ is the dual definable category of $\mc{T}$. For example, in the above theorem, the cotilting module $F$ is necessarily of the form $\Hom_{R}(T,E)$, where $E$ is an injective cogenerator in $\Mod{R}$ and $T$ is a tilting module. As $F$ is flat $T$ must be injective, and the tilting class it induces is the Gorenstein injective $R$-modules.

Recall that if $\mf{a}$ is an ideal of a commutative noetherian ring $R$, then the $i$-th local cohomology functor with support in $\mf{a}$ is defined to be 
\[
H_{\mf{a}}^{i}(-):=\rlim_{t\geq 0}\t{Ext}_{R}^{i}(R/\mf{a}^{t},-).
\]

\begin{lem}
	Let $R$ be a complete Gorenstein local ring of dimension $d$. If $C$ is a cotilting module for $\mc{GF}(R)$, then $H_{\mf{m}}^{d}(C)$ is a partial tilting module which is Gorenstein injective. 
\end{lem}

\begin{proof}
From the completeness assumption on $R$, local duality \cite[3.5.8]{bh} gives an isomorphism $\Hom_{R}(M,R)\simeq\Hom_{R}(H_{\mf{m}}^{d}(M),E(k))$ for all $R$-modules $M$. By \cite[3.9]{bird}, if $M$ is a Gorenstein flat module, then so is $\Hom_{R}(M,R)$, and therefore $H_{\mf{m}}^{d}(M)$ is a Gorenstein injective module by properties of dual definable categories. Suppose that $C$ is a cotilting module for $\mc{GF}(R)$, which necessarily has $H_{\mf{m}}^{d}(C)\neq 0$ by \cref{bbcmcotilting}, then $H_{\mf{m}}^{d}(C)$, and all coproducts of it, have finite injective dimension by \cite[2.5]{zz}, and as $R$ is Gorenstein this is equivalent to it (and all its coproducts) having finite projective dimension. Moreover, $\mc{GI}$ is closed under arbitrary coproducts as it is definable, and therefore $\t{Ext}_{R}^{i}(H_\mf{m}^{d}(C),H_{\mf{m}}^{d}(C)^{(\kappa)})=0$ for all cardinals $\kappa$ by \cite[11.2.2]{rha}. This completes the proof.
\end{proof}

We now relate the objects in $H_{\mf{m}}^{d}(C)^{\perp_{\infty}}$ with those in $^{\perp_{\infty}}C$. We let $(-)^{+}=\Hom_{R}(-,E(k))$ denote the usual Matlis duality functor.

\begin{lem}\label{GInjInclusion}
Suppose $R$ is a complete local Gorenstein ring and $C$ is a cotilting module for $\mc{GF}$. If $M$ is an $R$-module, then $M\in\,^{\perp_{\infty}}C$ (that is $M$ is Gorenstein flat) if and only if $M^{+}\in H_{\mf{m}}^{d}(C)^{\perp_{\infty}}$. In particular, $H_{\mf{m}}^{d}(C)^{\perp_{\infty}}$ contains $\mc{GI}$.
\end{lem}

\begin{proof}
	Let $M$ be an $R$-module. There is an isomorphism $H_{\mf{m}}^{d}(C)\simeq C\otimes_{R}E(k)$ by, for example, \cite[9.7]{24}; in particular, as $C$ is Gorenstein flat we have $C\otimes_{R}E(k)\simeq C\otimes_{R}^{\tb{L}}E(k)$ in $\mathsf{D}(R)$. This gives isomorphisms
\begin{align*}
	\t{Ext}_{R}^{i}(H_{\mf{m}}^{d}(C),\Hom_{R}(M,E(k))) & \simeq H_{-i}\,\tb{R}\Hom_{R}(C\otimes_{R}^{\tb{L}}E(k),\tb{R}\Hom_{R}(M,E(k))) \\
	&\simeq H_{-i}\,\tb{R}\Hom_{R}(M,\tb{R}\Hom_{R}(C\otimes_{R}^{\tb{L}}E(k),E(k))) \\
	&\simeq H_{-i}\,\tb{R}\Hom_{R}(M,\tb{R}\Hom_{R}(\tb{R}\Hom_{R}(C,E(k)),E(k)))\\
	&\simeq \t{Ext}_{R}^{i}(M,C^{++})
\end{align*}
Now, suppose that $M$ is Gorenstein flat. Then $M=\rlim_{j}M_{j}$ for a directed system $\{M_{j}\}_{J}$ of finitely generated Gorenstein projective modules. 

As the modules of finite injective dimension are definable, $C^{++}$ also has finite injective dimension, and thus $\t{Ext}_{R}^{i}(M_{j},C^{++})=0$ for all $i\geq 0$ by \cite[10.2.6]{rha}. Yet as $C^{++}$ is pure injective, we have
\[
\t{Ext}_{R}^{i}(M,C^{++})\simeq \llim_{J}\t{Ext}_{R}^{i}(M_{j},C^{++})=0
\]
for all $i>0$.

Thus $M\in\mc{GF}(R)$ gives $M^{+}\in H_{\mf{m}}^{d}(C)^{\perp_{\infty}}$. On the other hand, if $M^{+}\in H_{\mf{m}}^{d}(C)^{\perp_{\infty}}$, we have $\t{Ext}_{R}^{i}(M,C^{++})=0$ by the above isomorphisms. But $0\to C\to C^{++}$ is a pure-embedding with $C$ pure-injective, hence this is split. Consequently we have $\t{Ext}_{R}^{i}(M,C)=0$ for all $i>0$ as well, so $M$ is Gorenstein flat. 
\end{proof}
\begin{rem}
The above lemma shows that $(\mc{GF},H_{\mf{m}}^{d}(C)^{\perp_{\infty}})$ is a duality pair in the sense of \cite{hj}. Note that $(\mc{GF},\mc{GI})$ is also a duality pair, albeit one with significantly more structure.
\end{rem}

One may naturally wonder when the class $H_{\mf{m}}^{d}(C)^{\perp_{\infty}}$ coincides with the Gorenstein injective modules. This most certainly happens if $X\in H_{\mf{m}}^{d}(C)^{\perp_{\infty}}$ implies $X^{++}\in H_{\mf{m}}^{d}(C)^{\perp_{\infty}}$, since then we can deduce that $X^{+}\in\mc{GF}$ and thus $X\in\mc{GI}$. This occurs if, for example, $H_{\mf{m}}^{d}(C)^{\perp_{\infty}}$ is closed under pure submodules. In general, it is not clear that the dual of an object in $H_{\mf{m}}^{d}(C)^{\perp_{\infty}}$ is Gorenstein flat, that is in $^{\perp_{\infty}}C$. 

Having considered local cohomology, it is also natural to consider the canonical dual functor $\Hom_{R}(-,R)$, which is an endofunctor on $\mc{GF}(R)$, and $R$ is an injective cogenerator in the class of balanced big Cohen-Macaulay $R$-modules, as shown in \cite{bird}. 

\begin{thm}\label{cancot}
	Let $R$ be a complete Gorenstein local ring. If $C$ is a cotilting module inducing $\mc{GF}(R)$, then $C^{*}:=\Hom_{R}(C,R)$ is a partial cotilting module with complement, $Y$, such that $^{\perp_{\infty}}(C^{*}\oplus Y)=\,^{\perp_{\infty}}C=\mc{GF}(R)$. 
\end{thm}

\begin{proof}
Since $H_{\mf{m}}^{d}(C)$ is partial tilting, it has finite projective dimension, and thus by local duality $\Hom_{R}(C,R)\simeq\Hom_{R}(H_{\mf{m}}^{d}(C),E(k))$ has finite injective dimension. For the Ext-vanishing condition, recall from \cite[3.9]{bird} that $C^{*}$ is also a Gorenstein flat module. As $\mc{GF}(R)$ is definable $(C^{*})^{\kappa}$ is also Gorenstein flat and can therefore be expressed as $\rlim_{J}X_{j}$ for some finitely presented Gorenstein projective modules $X_{j}$. Then by the, at this stage familiar, isomorphisms $\t{Ext}_{R}^{i}(\rlim_{J}X_{j},C^{*})\simeq \llim_{J}\t{Ext}_{R}^{i}(X_{j},C^{*})$ from \cite[6.28]{gt}, it follows from \cite[11.5.9]{rha} that these Ext modules vanish, as $C^{*}$ has finite flat dimension. Thus $C^{*}$ is partial cotilting. We now show the latter claims. For this, it is enough to show that $^{\perp_{\infty}}C^{*}=\mc{GF}$, as the definability of $\mc{GF}$ yields the remaining claims by \cite[6.3]{ks}. Since $\Hom_{R}(C,R)\simeq \tb{R}\Hom_{R}(C,R)$ in $\mathsf{D}(R)$, we have the following isomorphisms in $\mathsf{D}(R)$:
\begin{align*}
\tb{R}\Hom_{R}(X,\tb{R}\Hom_{R}(C,R)) & \simeq \tb{R}\Hom_{R}(X,\tb{R}\Hom_{R}(H_{\mf{m}}^{d}(C),E(k))) \t{ by local duality }\\
&\simeq \tb{R}\Hom_{R}(H_{\mf{m}}^{d}(C),\tb{R}\Hom_{R}(X,E(k))) \t{ by adjunction }.
\end{align*}
Therefore $\t{Ext}_{R}^{i}(X,C^{*})=0$ for all $i>0$ if and only if $\t{Ext}_{R}^{i}(H_{\mf{m}}^{d}(C),X^{+})=0$ for all $i>0$. In particular, we see that $X\in\,^{\perp_{\infty}}C^{*}$ if and only if $X\in \mc{GF}$ by \cref{GInjInclusion}, which is what we required.
\end{proof}

\bibliographystyle{plain}
\bibliography{reference}

\end{document}